\documentclass[12pt, a4paper, parskip=half, abstracton]{scrartcl}

\usepackage{array}
\usepackage{marginnote}
\usepackage{xcolor}
\usepackage{amscd,amssymb,amsfonts,amsmath,latexsym,amsthm}
\usepackage{hyperref}
\usepackage[all,cmtip]{xy}
\usepackage{mathrsfs}
\usepackage{graphicx}
\usepackage{bm} 
\usepackage[nameinlink, capitalise]{cleveref}

\usepackage{etoolbox}

\def\hB{\hspace*{\fill}$\qed$}

\usepackage[nottoc]{tocbibind} 

\usepackage{defs_pp1}
\usepackage{slashed}
\usepackage[utf8]{inputenc}
\usepackage{microtype}
\usepackage[english]{babel}
\usepackage{mathtools}
\usepackage{bm}
\usepackage{esvect}

\title{Finite asymptotic dimension and the coarse  assembly map}
\author{
Ulrich Bunke\thanks{Fakult{\"a}t f{\"u}r Mathematik,
Universit{\"a}t Regensburg,
93040 Regensburg,
ulrich.bunke@mathematik.uni-regensburg.de} 
}

\numberwithin{equation}{section}
\newtheorem{theorem}{Theorem}[section] 
\newtheorem{prop}[theorem]{Proposition}
\newtheorem{lem}[theorem]{Lemma}

\newtheorem{ddd}[theorem]{Definition}

\theoremstyle{remark}
\theoremstyle{definition}

\newtheorem{ex}[theorem]{Example}
\newtheorem{rem}[theorem]{Remark}

\newcommand{\alg}{\mathrm{alg}}

\newcommand{\ee}{\mathrm{e}}

\newcommand{\EE}{\mathrm{E}}

\newcommand{\UBC}{\mathbf{UBC}}

\newcommand{\Yo}{\mathrm{Yo}}

\newcommand{\Hilb}{\mathbf{Hilb}}

\newcommand{\BC}{\mathbf{BC}}

\newcommand{\Cofib}{\mathrm{Cofib}}

\newcommand{\incl}{\mathrm{incl}}

\newcommand{\UK}{\mathrm{UK}}

\newcommand{\bF}{{\mathbf{F}}}

\newcommand{\cO}{{\mathcal{O}}}
\newcommand{\cU}{{\mathcal{U}}}

 \newcommand{\Cat}{{\mathbf{Cat}}}

\newcommand{\lf}{\mathrm{lf}}

\newcommand{\bP}{\mathbf{P}}

\newcommand{\KK}{\mathrm{KK}}

\newcommand{\cone}{\mathrm{cone}}

\newcommand{\nCcat}{C^{*}\mathbf{Cat}^{\mathrm{nu}}}

\renewcommand{\loc}{\mathrm{loc}}

\newcommand{\exa}{\mathrm{ex}}

\newcommand{\str}{\mathrm{str}}
\newcommand{\coarse}{\mathrm{coarse}}
\newcommand{\disc}{\mathrm{disc}}

\begin{document}  \maketitle \begin{abstract} 	
  In this note we  give a simple argument for the fact  that the coarse assembly map 
  for a strong coarse homology theory  with weak transfers and
  a bornological coarse space of weakly finite homotopical asymptotic dimension is a phantom equivalence.
  \end{abstract}

  \tableofcontents

  \section{Introduction}

  It is a classical result in coarse geometry that the coarse Baum-Connes conjecture holds for proper metric spaces
  of finite asymptotic dimension \cite{Yu_1998}. The proof given in  \cite{Yu_1998}   uses  the analytic details of the 
  construction of coarse $K$-homology in terms of $C^{*}$-algebras.  
  In order to decouple the internal details of the construction of coarse homology theories from the verification that the coarse assembly map is an equivalence in \cite{buen} we introduced the formalism of bornological coarse spaces and coarse homology theories. Using this formalism and axiomatizing an argument of \cite{nw1}   it was shown in \cite{ass} that the coarse assembly map for any strong coarse homology theory and a bornological coarse space of weakly finite asymptotic dimension is an equivalence. Applied to the case of the coarse $K$-homology  this  recovers  the result of  \cite{Yu_1998} using arguments in  coarse homotopy theory. The most technical part of the argument in \cite{ass}  is the verification  
 of the flasqueness of the coarsening space \cite[Thm. 5.55]{buen}. 
 
 The goal of this note is to present a  proof for the fact  that the coarse assembly map 
 for a strong coarse homology theory  with weak transfers and
 a bornological coarse space of weakly finite homotopical asymptotic dimension is a phantom equivalence.
 It avoids the usage of   \cite[Thm. 5.55]{buen}, but requires the additional assumption of the existence of  weak transfers on the coarse homology theory, which is fortunately satisfied by all examples of coarse homology theories of $K$-theoretic nature (topological or algebraic).  Also the statement is slightly weaker as we only show that the coarse assembly is a phantom equivalence.  
 On the other side, our argument  allows to weaken the assumption on $X$
 from weakly  finite asymptotic dimension to weakly finite homotopical asymptotic dimension.

 In \cref{okgpwergwrw} we  recall the  construction of the coarse assembly map from \cite{ass}
 and state the main result.
 In the  \cref{okgpwergwrw1} we introduce the notion of weakly finite homotopical asymptotic dimension
and prove the main theorem. 
 
 These notes grew out of an attempt to present a complete proof of the coarse Baum-Connes conjecture
under the finite asymptotic dimension assumption  in a lecture  course on coarse homotopy theory at the University of Regensburg in fall 2024.
  
  {\em Acknowledgement: The author was supported by the SFB 1085 (Higher Invariants) funded by the Deutsche Forschungsgemeinschaft (DFG).}

\section{The coarse assembly map}\label{okgpwergwrw}

In this section we introduce the coarse assembly map and state the main theorem.
We will assume familiarity with  the categories $\BC$ of bornological coarse spaces and $\UBC$ of uniform bornological coarse spaces
\cite{buen}, \cite{ass} and the  respective notions of  homology theories.    A coarse homology theory
  is a functor $E:\BC\to \cC$ to a cocomplete stable $\infty$-category $\cC$ which is coarsely invariant, excisive, $u$-continuous and vanishes on flasques, see  \cite[Sec. 4.4]{buen} for details. As explained in \cite[Sec. 4]{buen} there exists a universal 
 coarse homology theory $$\Yo:\BC\to \Sp\cX$$ such that  the pull-back along $\Yo$ identifies
 $\cC$-valued coarse homology theories with colimit-preserving functors $\Fun^{\colim}(\Sp\cX,\cC)$.
 We will use the same name for the coarse homology theory and the corresponding colimit-preserving functor. 
 
 A coarse homology theory is called strong if it  in addition annihilates weakly flasque bornological coarse spaces \cite[Def. 4.19]{equicoarse}. Likewise there exists  a universal strong coarse homology theory $$\Yo^{\str}:\BC\to \Sp\cX^{\str}\ .$$
 Since every strong coarse homology theory is in particular  a coarse homology theory we have a canonical
 colimit-preserving functor $$\iota^{\str}:\Sp\cX\to \Sp\cX^{\str}$$ classifying $\Yo^{\str}$.

A local homology theory is a functor $F:\UBC\to \cC$  to a cocomplete   stable $\infty$-category $\cC$ which is  homotopy invariant, (closed)-excisive, $u$-continuous and vanishes on flasques \cite[Def. 3.12]{ass}. Similarly as for coarse homology theories  there exists  a universal  local homology theory $$\Yo^{\loc}:\UBC\to \Sp\cB$$ such that the pull-back along $\Yo^{\loc}$ identifies
 $\cC$-valued local  homology theories with colimit-preserving functors $\Fun^{\colim}(\Sp\cB,\cC)$. 
 We will use again same name for the  local  homology theory and the corresponding colimit-preserving functor. 
 
 The categories of uniform bornological coarse spaces and bornological coarse spaces are related by a forgetful functor
\begin{equation}\label{dfgbdbdgdb}\iota:\UBC\to \BC\ .
\end{equation}
 
 \begin{ex} The compositions $$\bF:=\Yo \circ \iota:\UBC\to \Sp\cX\ , \qquad \bF^{\str}:=\Yo^{\str}\circ \iota:\UBC\to \Sp\cX^{\str}$$ are  local homology theories. \hB
 \end{ex}

 Besides the forgetful functor from \eqref{dfgbdbdgdb} we have the cone and the cone-at-$\infty$-functors $$\cO,\cO^{\infty}:\UBC\to \BC\ ,$$ see \cite[Sec. 8]{ass}. \begin{rem}
In this remark we recall their construction  Note that $\BC$ and $\UBC$ have symmetric monoidal structures $\otimes $.
 The cone $\cO^{\infty}(X)$ over a uniform bornological coarse space $X$ is  the bornological  coarse space obtained from the $\R\otimes \iota X$ by equipping it with the hybrid coarse structure associated to the big family of subsets $((-\infty,n]\times X)_{n\in \nat}$, see \cite[Def. 5.10]{buen}. Thus a coarse entourage of $\R\otimes \iota X$ is a coarse entourage of $\cO^{\infty}(X)$ if and only if for every uniform entourage $Z$ of $\R\otimes X$ there exists
 $n$ in $\nat$ such that $U\cap ((n,\infty)\times X)^{2}\subseteq Z$. The cone $\cO(X)$ is the subset $[0,\infty)\times X$ of $\cO^{\infty}(X)$ with the induced structure.
 We have a natural sequence of morphisms of bornological coarse spaces
\begin{equation}\label{gewrfwerfwerfwf}\iota X\xrightarrow{x\mapsto (0,x)} \cO(X)\xrightarrow{\incl} \cO^{\infty}(X)\xrightarrow{\partial^{\cone}} \R\otimes \iota X
\end{equation} 
 where $\partial^{\cone}$ is given by the identity of underlying sets and called the cone boundary. \hB
 \end{rem}
 Applying    $\Yo^{\str}$  to the sequence \eqref{gewrfwerfwerfwf} and letting $X$ vary we get
 the cone sequence \cite[(9.1)]{ass} \begin{equation}\label{grewofijowerpfwer}\bF^{\str} \to \Yo^{\str}\cO\to \Yo ^{\str}\cO^{\infty}\stackrel{\partial^{\cone}}{\to} \Sigma \bF^{\str}
\end{equation}
  of local homology theories which we can also view  as a fibre  sequence of  colimit preserving functors from $\Sp\cB$ to $\Sp\cX^{\str}$.
 
 The Rips complex functor $(X,U)\mapsto P_{U}(X)$ sends a bornological coarse space $X$ with a coarse entourage $U$   in $\cC_{X}$ containing the diagonal to a uniform bornological coarse space. Thereby $P_{U}(X)$  is the simplical space of finitely supported $U$-bounded probability measures on $X$ with the coarse and uniform structure induced by the spherical path metric on the simplices, and the minimal compatible bornology such that the inclusion $X\to P_{U}(X)$ by Dirac measures is bornological. By composing with $\Yo^{\loc}$ and forming the colimit over the entourages  (more precisely using left Kan-extensions, see \cite[Def. 5.1]{ass})
 we get a functor
$$\bP:\BC\to \Sp\cB$$ which turns out to be  a coarse homology theory.

 The inclusions $X\to P_{U}(X)$ by Dirac measures for all $X$ and $U$  induce a natural equivalence $\iota^{\str}\stackrel{\simeq}{\to} \bF^{\str}\circ \bP$ of functors from $\Sp\cX$ to $ \Sp\cX^{\str}$.

%
%

Precomposition  of the cone sequence \eqref{grewofijowerpfwer} with $\bP$ and   the equivalence $\iota^{\str}\simeq \bF^{\str}\circ \bP$ yields the fibre sequence of 
$\Sp\cX^{\str}$-valued coarse homology theories
 \begin{equation}\label{herthrtgertge} \iota^{\str} \to \Yo^{\str}\cO\bP\to \Yo ^{\str}\cO^{\infty}\bP\stackrel{\partial^{\coarse}}{\to} \Sigma \iota^{\str}\ ,
\end{equation}
 i.e., a fibre sequence of colimit preserving  functors from $\Sp\cX$ to $\Sp\cX^{\str}$.
   
   \begin{ddd}The map $\partial^{\coarse}$ is called the universal coarse assembly map.\end{ddd}
 We  postcompose with a strong coarse homology theory $E:\Sp\cX^{\str}\to \cC$ 
 and evaluate the sequence \eqref{herthrtgertge} at an object $X$ in $\Sp\cX$. Following \cite[Def. 9.7]{ass} we adopt the following:
 
\begin{ddd}\label{kogpewgwgfwf}
The map
\begin{equation}\label{fwefdwdqw}\mu^{\coarse}_{E,X}:E\cO^{\infty}\bP(X)\to \Sigma E(\iota^{\str}X)
\end{equation} 
  induced by $\partial^{\coarse}$ is called the coarse assembly map for $E$ and $X$.
  \end{ddd}

\begin{rem}
We refer to \cite{ass} for a detailed discussion of the domain of the coarse assembly map. Recall that 
 the coarse homology theory $E$ is called additive if the target $\cC$ is also complete and  we have $E(I_{min,min})\simeq \prod_{i\in I} E(\{i\})$
for any set $I$.  If $W$ is homotopy equivalent in $\UBC$ to a finite-dimensional simplicial complex with
the spherical path metric and $E$ is additive, then $E\cO^{\infty}(W)$ is  by \cite[Prop. 11.23]{ass} equivalent to $\Sigma^{-1}E(*)^{\lf}(W)$, the locally finite homology of $W$ with coefficients in
$\Sigma^{-1}E(*)$.  If $X$ is a bornological coarse space of bounded geometry, then $P_{U}(X)$ is finite dimensional for every $U$ in $\cC_{X}$. In this case   we can  express the domain of \eqref{fwefdwdqw}
 as  $$E\cO^{\infty}\bP(X)\simeq\colim_{U\in \cC_{X}} \Sigma^{-1}E(*)^{\lf}(P_{U}(X))\ .$$
 This is close to the   the classical formula for the domain of the coarse assembly map, see e.g. \cite[Sec. 5.6]{roe_lectures_coarse_geometry}, \cite[Sec. 2]{hr}.
 \hB
\end{rem}

\begin{rem}
The coarse Baum-Connes assembly map appearing in the coarse Baum-Connes conjecture \cite{hr}  corresponds to  the coarse assembly map $$\mu_{K\cX,X}^{\coarse}:K\cX\cO^{\infty}\bP(X)\to \Sigma K\cX(\iota^{\str}X)$$ for the topological coarse $K$-homology theory $K\cX:\BC\to \Mod(KU)$  constructed in detail in \cite[Sec. 8]{buen},  \cite{coarsek}. In the classical   literature the coarse $K$-homology   is defined  in terms of Roe algebras   as a $\Z$-graded group-valued functor on  proper metric spaces $X$.
The details of the classical coarse Baum-Connes assembly map  are  very analytical. In the work of  \cite{hr} 
the domain of the assembly map is expressed  in terms of an analytic model of the locally finite $K$-homology like
$\KK_{*}(C_{0}(Y),\C)$ and Paschke duality. 
Alternatively,  in \cite{Yu_1998}, \cite{zbMATH01388911} the domain of the coarse Baum-Connes assembly map is expressed in terms of  Yu's  localization algebra.
  Here $Y$ is a metric space, e.g. the Rips complex $P_{U}(X)$ if $X$ has bounded geometry.
 \hB\end{rem}
  
The main point of   \cref{kogpewgwgfwf} above is to write the coarse assembly map
as a component of a natural transformation between coarse homology theories defined on all of $\BC$. The strong coarse homology theory $E$ can be considered as a variable.
 In our approach the internal details of the construction of $E$
are not relevant.

In order to state our main theorem we need some notions from higher category theory.
Let $\cC$ be a cocomplete stable $\infty$-category.  
\begin{ddd}[D. Clausen]\label{kopwegregwefw}\mbox{}\begin{enumerate}\item\label{elgpwegewrfw}
A morphism $C\to D$ in $\cC$ is compact if for every filtered system $(E_{i})_{i\in I}$ in $\cC$ with $\colim_{i\in I}E_{i}\simeq 0$ and map $D\to E_{i}$ for some $i$ in $I$ there exists $i\to j$ in $I$  such that
$C\to D\to E_{i}\to E_{j}$ vanishes.
\item The category $\cC$ called is compactly assembled if $\cC\simeq \Ind_{\aleph_{1}}(\cC^{\aleph_{1}})$ and
every object of   $\cC^{\aleph_{1}}$ is the colimit of an $\nat$-indexed system with compact structure maps, where  $\cC^{\aleph_{1}}$ is the subcategory of $\aleph_{1}$-compact objects. 
\end{enumerate}
\end{ddd}

\begin{ddd}\mbox{}
\begin{enumerate}
\item A morphism  $D\to E$ in $\cC$ is a phantom morphism if for every compact morphism $C\to D$ the composition $C\to D\to E$ vanishes.

\item 
An object $E$ in $\cC$  is called a phantom object if $\id_{E}$ is a phantom morphism.   \item A morphism $C\to D$ in $\cC$ is a phantom equivalence if its fibre is a phantom object.\end{enumerate}
\end{ddd}

\begin{ex}\label{okwegpwegerw9}
 Note that $\cC$ is the target category of the coarse homology theory $E$ under consideration.
If $\cC$ is compactly generated or compactly assembled, then   phantom objects in $\cC$ are zero objects. 
In this case phantom equivalences are equivalences.

Examples of coarse homology theories to which our results apply are the coarse algebraic $K$-homology
$K \cX^{\alg}_{\bC}$ with coefficients in a left-exact  $\infty$-category $\bC$ introduced in \cite{unik} or the coarse topological $K$-theory
$K\cX_{\bC}$ with coefficients in a $C^{*}$-category $\bC$ constructed in \cite{coarsek}, where we set $K\cX:=K\cX_{\Hilb_{c}}(\C)$. In these examples the target categories are  spectra $\Sp$ or of $KU$-modules $\Mod(KU)$, respectively. They are both compactly generated.

Note that the functors $K \cX^{\alg}_{\bC}$ and $K\cX_{\bC}$ above are constructed by  composing  the algebraic or topological 
$K$-theory functor with a controlled object functor $\bV_{\bC}$ from $\BC$ to $\Cat^{\exa}_{\infty}$ or $\nCcat$, respectively. If one replaces algebraic $K$-theory by
the universal finitary localizing invariant $\cU_{\loc}:\Cat^{\exa}_{\infty}\to \cM_{\loc}$ \cite{MR3070515} or, likewise  the topological $K$-theory by the  $E$-theory functor $\ee:\nCcat\to \EE$, then one obtains
coarse homology theories $\UK\cX_{\bC}:=\cU_{\loc}\circ \bV_{\bC}:\BC\to \cM_{\loc}$ (see \cite[Sec. 5.3]{unik}), or $\EE\cX_{\bC}:=\ee\circ \bV_{\bC}:\BC\to \EE$ (see \cite[Rem. 2.59]{Bunke:2024aa}), respectively which also satisfy the assumptions in \cref{jifofwqewfwf} below.
By a result of Efimov  we know that the stable $\infty$-category of localizing motives $\cM_{\loc}$ is compactly assembled. 
The stable $\infty$-category $E$  representing $E$-theory  is compactly assembled by \cite{budu}.
 \hB
\end{ex}
 
The main goal of this note is to show:
 
\begin{theorem}\label{jifofwqewfwf}
If the bornological coarse space $X$ has  weakly finite homotopical asymptotic dimension and the strong coarse homology theory $E$ admits weak transfers, then $\mu^{\coarse}_{E,X}$ is a phantom equivalence.   
\end{theorem}

We refer to \cref{kgoperwergrefwf} and \cref{wkgpwegwergw9} for the notions of weakly finite homotopical asymptotic dimension and weak transfers. The proof of the theorem will be given at the end of \cref{okgpwergwrw1}.

\begin{rem}
\cref{jifofwqewfwf} is not the best possible result. 
In \cite[Thm. 10.4]{ass} (by axiomatizing arguments  of  \cite{nw1}) we have shown that if $X$ has  weakly finite asymptotic dimension (in the classical sense) and $E$ is any  strong coarse homology theory, then $E\cO\bP(X)\simeq 0$ and hence $\mu_{E,X}$ is actually an equivalence. 

Under the assumption on $E$ made in \cref{jifofwqewfwf},  by the main result of  \cite{Bunke_2019}  one can  replace the assumption on $X$ by 
  finite decomposition complexity.   
 Finite asymptotic dimension (in the classical sense) implies finite decomposition complexity, but the relation
 between   finite decomposition complexity and weakly finite homotopical asymptotic dimension remains unclear.

It was shown in   \cite{zbMATH01388911}  that the classical coarse Baum-Connes assembly map (this corresponds to the case $E=K\cX$)   is an isomorphism for  proper metric spaces  $X$  bounded geometry which admit a coarse embedding into a Hilbert space. 
The arguments for this result are of  analytic nature and strongly use the details of topological $K$-theory and $C^{*}$-algebras associated to proper metric spaces. 
It is not clear wether there is a version of \cref{jifofwqewfwf} for general coarse homology theories $E$
under the assumption that $X$ has a coarse embedding into a Hilbert space.
%
%
\hB \end{rem}

\begin{rem}
If one applies \cref{jifofwqewfwf} to the topological coarse $K$-homology $K\cX$ and a proper metric space $X$ of finite asymptotic dimension, then one recovers the result of \cite{Yu_1998}. 

The coarse algebraic homology $K\cX^{\alg}_{\Mod(R)}$
 with coefficients in the category of $R$-modules for some  ring $R$  
 is just another example of a strong coarse homology with weak transfers. Therefore
  \cref{jifofwqewfwf} implies that the coarse assembly map  $\mu_{K\cX^{\alg}_{\Mod(R)},X}$ is an   equivalence  for proper metric spaces $X$ of finite asymptotic dimension.
 This  statement  could be considered as a formalized version of  results of \cite{Carlsson_2004} or  \cite{Bartels_2003}.
 The corresponding generalization to spaces with finite decomposition complexity
  is due to \cite{Guentner:2010aa}, whose  formalized version  can also be obtained by specializing   \cite{Bunke_2019}.
   A precise comparison of \cref{jifofwqewfwf} with these results for algebraic $K$-theory
  is complicated since the main goal of these papers was to show the injectivity
  of the assembly map for algebraic $K$-theory of group rings using the descent principle. For them the statements about the coarse assembly
  maps were technical steps in a longer proof. The were formulated directly for geometric  models  of algebraic 
  $K$-theory  and not formalized in the abstract language of coarse homology theories. 
     \hB

\end{rem}

\section{Weakly finite homotopical asymptotic dimension}\label{okgpwergwrw1}

Let $W$ be a simplicial complex which we consider as a uniform bornological coarse space with coarse and uniform structure induced by the spherical path metric and some compatible bornology. We further consider  
a big family
 $\cV:=(V_{n})_{n\in \nat}$  of subcomplexes
 with the property that every bounded subset of $W$ is contained in   $V_{n}$ for some $n$ in $\nat$. We let $W_{h}$ be the bornological coarse space obtained from $W$ by equipping it with the hybrid coarse structure \cite[Def. 5.10]{buen} associated to the big family $\cV$. 
By definition a coarse entourage of $W$ belongs to the hybrid structure if for every uniform entourage $Z$
of $W$ there exists an $n$ in $\nat$ such that $U\cap (W\setminus V_{n})^{2}\subseteq Z$.
 Note that the bornological coarse structure induced by $W_{h}$ on $V_{n}$ coincides with the bornological coarse structure induced by $W$.

We let $\Sp\cX\langle \disc \rangle$ denote the localizing subcategory of $\Sp\cX$  generated by $\Yo(X )$ for all  bornological coarse spaces
$X$ with the  discrete coarse structure.  We let \begin{equation}\label{fewfeqwdqdwedq}\Yo(W_{h},\cV):=\colim_{n\in \nat} \Cofib(\Yo(V_{n})\to \Yo(W_{h}))\ .
\end{equation} 
The following has been shown in \cite[Thm. 5.57]{buen}.
\begin{prop}\label{kgopwergrefwef} If $\dim(W)<\infty$, then we have
  $\Yo(W_{h},\cV)\in \Sp\cX\langle \disc \rangle$.
\end{prop}\begin{proof}
The idea is to argue by induction on the dimension of $W$.
If $\dim(W)=0$, then $W_{h}$ is discrete and there is nothing   to  show.
If $\dim(W)=k$, then we write $W$ as a union of a thickening of the  $k-1$-skeleton
and a disjoint collection of $k$-simplices, see the pictures on \cite[p. 80]{buen}.  The decomposition theorem  \cite[Thm. 5.22]{buen}
presents  $ \Yo(W_{h},\cV)$ as a push-out.  The thickenings of the $k-1$-skeleton and the intersection
are homotopy equivalent in $\UBC$ to $k-1$-dimensional simplicial complexes, and the collection of $k$-simplices is homotopy equivalent to the   discrete uniform space $B$ of barycenters.  We then use the homotopy theorem \cite[5.26]{buen} in order to invoke the induction hypothesis and the following observation.

The  coarse structure induced on the set $B$ of barycenters is not the discrete one. But the relative 
coarse homology $\Yo(B,\cV\cap B)$ is actually coarsening invariant, i.e., it does not depend on the coarse structure as long as the latter is compatible with the uniform structure. Since $B$ has the discrete uniform structure we
can take the discrete coarse structure on $B$ as well and then conclude that $ \Yo(B,\cV\cap B)\in \Sp\cX\langle \disc \rangle$.
\end{proof}

Let $\bE$ be any strong coarse homology theory.

\begin{lem}\label{kogpwergwrefwf}
$E\cO\bP$ vanishes on $ \Sp\cX\langle \disc \rangle$.
\end{lem}
\begin{proof} It suffices to show that
$E\cO\bP(X)\simeq 0$  for all discrete bornological coarse spaces $X$.
 Then $\bP(X)$ is the uniform bornological coarse space $X_{\disc}$ obtained from  $X$ by equipping it with the discrete uniform structure.
We have an isomorphism $\cO(X_{\disc})\cong X\otimes \cO(*)$. Since the latter bornological coarse space is flasque we get $E\cO\bP(X)\simeq E(\cO(X_{\disc}))\simeq E(X\otimes \cO(*))\simeq 0$.
\end{proof}

 If $(W_{n})_{n\in \nat}$ is a family of simplicial complexes, then we form $W:=\bigsqcup_{n\in \nat} W_{n}$
with the big family $\cV:=(V_{n})_{n\in \nat}$ given by $V_{n}:=\bigcup_{m\le n}W_{m}$.

Note that $\BC$ has a symmetric monoidal structure $\otimes$ which descends to $\Sp\cX$.
We write $I_{min,min}$ for the set $I$ equipped with the discrete coarse structure and the minimal bornology.

Let $X$ be a bornological coarse space.

\begin{ddd}\label{elrhpgregeg}
We say that $X$ has finite homotopical asymptotic dimension  if there exists a sequence of  morphisms
$(\phi_{n}:X\to W_{n})$ in $\BC$ from $X$ to simplicial complexes $W_{n}$ such that
\begin{enumerate}
\item $\sup_{n\in \nat}\dim(W_{n})<\infty$.
\item  \label{grkoepwergrwegw} $\Yo(\phi_{n})$ is a split mono. \item \label{grkoepwergrwegw1} The map $\phi:=\sqcup_{n\in \nat} \phi_{n}:\nat_{min,min}\otimes X\to W_{h}$ is a morphism in $\BC$.
\end{enumerate}
\end{ddd}
Condition \cref{elrhpgregeg}.\ref{grkoepwergrwegw1} means that for every
entourage $U$ in $\cC_{X}$ and $\epsilon$ in $(0,\infty)$ there exists $n_{0}$ in $\nat$ such that  for all $n$ in $\nat$ with $n\ge n_{0}$ we have $\phi_{n}(U)\subseteq U_{\epsilon}$, where $U_{\epsilon}$ is the metric entourage of width $\epsilon$.

\begin{rem} We refer to \cite[Def. 9.4]{roe_lectures_coarse_geometry} for the classical definition of finite asymptotic dimension of a metric space  and to \cite[Thm. 9.9(d)]{roe_lectures_coarse_geometry} for a translation which is close in spirit to the definition above. 

The difference between the homotopical and classical notion is in  \cref{elrhpgregeg}.\ref{grkoepwergrwegw}, where for the classical notion we would require the stronger condition that $\phi_{n}$ is a coarse equivalence.
One could ask wether the homotopical  version of the finite asymptotic dimension condition is an interesting generalization.
\hB  
  \end{rem}

If $X$ is a bornological coarse space and $U$ is a coarse entourage of $X$, then by $X_{U}$ we denote the bornological coarse space obtained from $X$ by replacing the coarse structure by the in general smaller coarse structure generated by $U$.
\begin{ddd}\label{kgoperwergrefwf}
We say that $X$ has weakly finite homotopical asymptotic dimension if there exists a cofinal subset $\cC_{X}'\subseteq  \cC_{X}$   such that
$X_{U}$ has finite homotopical asymptotic dimension for all $U$ in $\cC_{X}'$.
\end{ddd}

\begin{rem}
Let $X$ be a bornological coarse space with weakly finite homotopical asymptotic dimension.
Then   the family of   simplicial complexes $(W_{U,n})_{n\in \nat}$
witnessing the finite homotopical asymptotic dimension of the spaces $X_{U}$ according to \cref{elrhpgregeg} could depend on $U$. In particular, the number
$\sup_{n\in \nat} \dim(W_{U,n})$ could
be unbounded if $U$ runs over the cofinal subset $\cC_{X}'$  mentioned in \cref{kgoperwergrefwf}. 
\hB
\end{rem}

\begin{ex}
We consider the set $X:=\bigsqcup_{n\in \nat}\R^{n}$ with the colimit coarse structure. In this structure the components are coarsely disjoint and the generating entourages 
are the metric entourages of the components. We equip $X$ with the  bornology  generated by the metrically bounded sets in the components. 

The bornological coarse space $X$ has weakly finite asymptotic dimension.  In order to see this we let the cofinal subset $\cC_{X}'$ consist of entourages $U_{n}$ which are the width-one  metric entourages on the components $\R^{m}$ for $m\le n$ and the diagonal on the components $\R^{m}$ with $m>n$. Then $X_{U_{n}}$ has asymptotic dimension $n$ in the classical sense. 

Using ordinary coarse homology one can check that  $X$ does not have finite homotopical asymptotic dimension. 

A  coarsely connected example of a bornological coarse space with weakly finite asymptotic dimension, but
not finite homotopical asymptotic dimension, is   $ \colim_{n\in \nat}\Z^{n}$ (with the standard inclusions $\Z^{n}\to \Z^{n+1}$) with the colimit coarse structure and the minimal bornology.
 \hB
 \end{ex}

 A set $I$ gives rise to an endofunctor $I_{min,min}\otimes-$ of $\BC$ or $\Sp\cX$.
Using excision, for $i$ in $I$ the decomposition $I=\{i\}\sqcup I\setminus\{i\}$ gives  rise to  natural transformation
 $\pr_{i}:I_{min,min}\otimes -\to  \id$ of endofunctors of $\Sp\cX$.

Let $E$ be a strong coarse homology theory. Following \cite[Def. 2.16]{Bunke_2019} we adopt the following:
\begin{ddd}\label{wkgpwegwergw9}
$E$ has weak transfers if for any $X$ in $\BC$ there exists
a map $$\tr_{X}:E(X)\to E(\nat_{min,min}\otimes X)$$ such that
the components $E(X)\stackrel{\tr_{X}}{\to} E(\nat_{min,min}\otimes X)\stackrel{\pr_{n}}{\to} E(X)$ are equivalent to the identity for all $n$ in $\nat$.
\end{ddd}

The obvious isomorphisms $P_{\diag(\nat)\times  U}(\nat_{min,min}\otimes X)\cong \nat_{\disc}\otimes P_{  U}(  X)$
for any $X$ in $\BC$ with entourage $U$
and $\cO(\nat_{\disc}\otimes Y)\cong \nat_{min,min}\otimes \cO(Y)$ for any $Y$ in $\UBC$
induce an equivalence
  $\Yo^{\str} \cO\bP(\nat_{min,min}\otimes X)\simeq \nat_{min,min}\otimes \Yo^{\str}\cO\bP(X)$.  Hence if $E$ has weak transfers, then so has   the coarse homology theory
 $ E\cO\bP$.

\begin{ex}
If $E$ has transfers in the sense of \cite{trans}, then it has weak transfers \cite[Rem. 2.18]{Bunke_2019}. Coarse topological $K$-homology  $K\cX_{\bC}$  with coefficients in some $C^{*}$-category $\bC$, or coarse algebraic $K$-homology $K\cX^{\alg}_{\bC}$ with coefficients in some left-exact $\infty$-category $\bC$, have transfers as shown in \cite{coarsek} or \cite{unik}, respectively. 
Likewise the universal versions $\UK\cX_{\bC}$ or $\EE\cX_{\bC}$  explained in \cref{okwegpwegerw9} have transfers.
\hB \end{ex}

\begin{proof}[Proof of Theorem \ref{jifofwqewfwf}]
By $u$-continuity we have $E\cO\bP(X)\simeq \colim_{U\in \cC_{X}} E\cO\bP(X_{U})$. By assumption, 
$X_{U}$ has finite homotopical asymptotic dimension for a cofinal set of $U$ in $\cC_{X}$. Since a filtered colimit of phantom objects is phantom it suffices to show that
$E\cO\bP(X)$ is phantom under the assumption that $X$ itself has finite homotopical asymptotic dimension.

The map $\phi$ below is taken from  \cref{elrhpgregeg}.\ref{grkoepwergrwegw1}.
We let $p:E\cO\bP( W_{h})\to E\cO\bP( W_{h},\cV)$ denote the canonical projection.
The target of the map
$$E\cO\bP(X)\stackrel{\tr_{X}}{\to}  E\cO\bP(\nat_{min,min}\otimes X)\stackrel{\phi}{\to} E\cO\bP( W_{h})\stackrel{p}{\to} E\cO\bP( W_{h},\cV)$$ vanishes by a combination of \cref{kgopwergrefwef} and \cref{kogpwergwrefwf}. If $\kappa:C\to E\cO\bP(X)$ is a compact map, then in view of \eqref{fewfeqwdqdwedq} and \cref{kopwegregwefw}.\ref{elgpwegewrfw} 
and by excision for the last equivalence  the composition
\begin{equation}\label{fwerfewrfwref}C\stackrel{\kappa}{\to} E\cO\bP(X)\stackrel{\phi\circ \tr_{X}}{\to}E\cO\bP( W_{h})\stackrel{p_{n-1}}{\to} E\cO\bP( W_{h},V_{n-1}) \simeq E\cO\bP( W_{h}\setminus V_{n-1})
\end{equation}  vanishes for some $n$ in $\nat$.
Again using excision we have a projection $$E\cO\bP( W_{h}\setminus V_{n-1})\simeq
E\cO\bP( W_{h}\setminus V_{n})\oplus E\cO\bP(W_{n})\to  E\cO\bP(W_{n})\ .$$
The composition of this projection with the map $p_{n-1}\circ \phi\circ \tr_{X}$ appearing in \eqref{fwerfewrfwref} is the map induced by the component $\phi_{n}$.
Therefore the  map $$C\stackrel{\kappa}{\to} E\cO\bP(X)\stackrel{\phi_{n}}{\to} E\cO\bP( W_{n})$$  vanishes. Since the second map is  split monomorphism 
by \cref{elrhpgregeg}.\ref{grkoepwergrwegw} we can conclude that $\kappa \simeq 0$. Since 
$\kappa$ was arbitrary we conclude that $ E\cO\bP(X)$ is a phantom object.
\end{proof}

%
%


    \bibliographystyle{alpha}
\bibliography{forschung2021}

\end{document}